\documentclass[a4paper]{amsart}
\usepackage[utf8]{inputenc}

\title[UMF of homeomorphism groups]{Universal minimal flows of homeomorphism groups of high-dimensional manifolds are not metrizable}

\author{Yonatan Gutman}
\address{Institute of Mathematics \\
  Polish Academy of Sciences \\
  ul. \'{S}niadeckich 8  \\
  00-656 Warszawa \\
  Poland}
\email{Y.Gutman@impan.pl}

\author{Todor Tsankov}
\address{
  Institut Camille Jordan \\
  Universit\'e Claude Bernard Lyon 1 \\
  Universit\'e de Lyon \\
  43, boulevard du 11 novembre 1918 \\
  69622 Villeurbanne \textsc{cedex} \\
  France
  -- and --
  Institut Universitaire de France}
\email{tsankov@math.univ-lyon1.fr}

\author{Andy Zucker}
\address{
  Institut de Math{\'e}matiques de Jussieu--PRG \\
  Universit\'e Paris Diderot, case 7012 \\
  8, place Aur{\'e}lie Nemours \\
  75205 Paris \textsc{cedex} 13 \\
  France}
\email{andrew.zucker@imj-prg.fr}

\date{October 2019}

\subjclass[2010]{Primary 54H20; Secondary 57N99}
\keywords{universal minimal flow, homeomorphism group, manifold, Hilbert cube, metrizablity}

\usepackage[T1]{fontenc}
\usepackage[osf]{mathpazo}
\usepackage{eucal}

\usepackage{hyperref}
\usepackage{my-macros}

\usepackage{enumitem}
\setlist[enumerate,1]{label=(\roman*), font=\normalfont}

\usepackage[initials, shortalphabetic]{amsrefs}

\newcommand{\im}[1]{\mathrm{Im}(#1)}

\begin{document}

\begin{abstract}
Answering a question of Uspenskij, we prove that if $X$ is a closed manifold of dimension $2$ or higher or the Hilbert cube, then the universal minimal flow of $\operatorname{Homeo}(X)$ is not metrizable. In dimension $3$ or higher, we also show that the minimal $\operatorname{Homeo}(X)$-flow consisting of all maximal, connected chains in $X$ has meager orbits.
\end{abstract}

\maketitle

\section{Introduction}
\label{sec:introduction}

A central object in abstract topological dynamics is the \emph{universal minimal flow (UMF)} of a topological group $G$, often denoted by $M(G)$. It is a canonical dynamical system associated to $G$ that is defined abstractly as the minimal $G$-flow\footnote{A \emph{$G$-flow} is a continuous action of $G$ on a compact Hausdorff space; it is \emph{minimal} if all orbits are dense.} that admits any minimal $G$-flow as a factor. It is unique with this property, up to isomorphism (see \cite{gutman2013new} for a short proof). For most groups encountered in classical mathematics (for example, discrete and locally compact, non-compact groups), the UMF is a non-metrizable space that is difficult to describe explicitly. Yet somewhat surprisingly, for many important infinite-dimensional Polish groups, it is a rather concrete object. The first examples of this phenomenon were \emph{extremely amenable groups}, for which the UMF is a singleton. For some time, those examples were considered pathological until it was realized that they are ubiquitous and that their study has deep connections with combinatorics and probability theory. We recommend Pestov's book \cite{Pestov2006} as an introduction to the subject.

The metrizability of the universal minimal flow of a Polish group $G$ is a dividing line between well-behaved and wild dynamics for the collection of all minimal $G$-flows. If the UMF of $G$ is metrizable, it has a comeager orbit and can be represented as the completion of a homogeneous space $G/H$ for a suitably chosen closed subgroup $H \leq G$ (\cites{BenYaacov2017,Melleray2016}, see also \cite{Zucker2018} for a different proof). This implies that all minimal $G$-flows are metrizable and have a comeager orbit and, up to isomorphism, they are \emph{concretely classifiable} (that is, isomorphism classes can be represented in a concrete way as points in a Polish space) \cite{Melleray2016}*{Theorem~3.5}. In the other direction, the techniques of \cite{BenYaacov2017} and \cite{Zucker2018} allow to show that the UMF of a given group $G$ is not metrizable by studying a single, sufficiently rich, metrizable $G$-flow (for example, by showing that all of its orbits are meager, or, more generally, by considering its \emph{maximal highly proximal extension}).

Recently, Kwiatkowska~\cite{Kwiatkowska2018} found an example that shows that the universal minimal flow can have a comeager orbit while not being metrizable. Groups with this property form an intermediate class between those with metrizable UMF and those for which the UMF has meager orbits. For instance, analogously to the metrizable case, one can find a closed subgroup $H\leq G$ such that $M(G)$ is the \emph{Samuel compactification} of $G/H$ \cite{Zucker2020}. This means that $M(G)$ is determined by the behavior of some Polish (albeit non-compact) $G$-invariant subspace. In contrast, if $M(G)$ has all orbits meager, no $G$-invariant subspace of $M(G)$ is Polish (as can be seen using the techniques from \cite{BenYaacov2017}).

The first interesting, non-trivial, metrizable universal minimal flow of a Polish group was computed by Pestov~\cite{Pestov1998}, who proved that the UMF of the group $\Homeo^+(S^1)$ of orientation-preserving homeomorphisms of the circle equipped with the compact-open topology is the circle itself. This naturally led him to ask the question whether a similar result is true for homeomorphism groups of other manifolds. A few years later, Uspenskij~\cite{Uspenskij2000} proved that the action of a group on its UMF is never $3$-transitive, thus giving a negative answer to Pestov's question. His proof is short but quite remarkable and led to many interesting developments in the subject. Given an arbitrary flow $G \actson X$, Uspenskij considers the action of $G$ on the space $\Phi(X)$ of \emph{maximal chains of closed subsets of $X$} (viewed as a subspace of the double hyperspace of $X$). He then proceeds to prove that if the action $G \actson X$ is $3$-transitive, then there is no $G$-map from $X$ to this space of maximal chains, showing that $X$ is not universal. This construction is particularly interesting for homeomorphism groups as they are already equipped with a natural action on a compact space. Inspired by it, Glasner and Weiss~\cite{Glasner2003a} showed that the UMF of the homeomorphism group of the Cantor space $2^\N$ is isomorphic to $\Phi(2^\N)$ and more generally, Glasner and Gutman~\cite{GG2012universal} showed that the UMF of the homeomorphism group of a h-homogeneous zero-dimensional compact Hausdorff
space $X$ is isomorphic to $\Phi(X)$. Still, for higher-dimensional (sufficiently homogeneous) compact spaces, it remained an open problem whether the UMFs of their homeomorphism groups are metrizable and Uspenskij~\cite{Uspenskij2000} specifically asked if this is case for  the homeomorphism group of the Hilbert cube whereas Pestov~\cite{Pestov2006}*{Open Problem~6.4.13} and \cite{pestov2007forty}*{Questions ~22 \& 23} asked for an explicit description of the UMF of the homeomorphism groups of both closed finite-dimensional manifolds and the Hilbert cube.

If the space $X$ is connected, it is easy to see that $\Phi(X)$ is not minimal: indeed, the subspace $\cC(X)$ of $\Phi(X)$ consisting of \emph{connected chains} is non-empty, closed and invariant. The UMF of $\Homeo(X)$ has been computed for several one-dimensional spaces $X$: the interval, the circle \cite{Pestov1998}, and the Lelek fan (Barto\v{s}ová--Kwiatkowska~\cite{Bartosova2019}) and in all of these cases it can be realized as a subspace of $\cC(X)$ in a natural way. A particularly intriguing open question in dimension $1$, again asked by Uspenskij in \cite{Uspenskij2000}, is whether the UMF of the homeomorphism group of the pseudo-arc $\bP$ is metrizable and, in particular, whether it is isomorphic to $\bP$. Note that as $\bP$ is hereditarily indecomposable, $\cC(\bP) \cong \bP$ (see Section~\ref{sec:comeager-set-chains} for more discussion).

It turns out that the situation is dramatically different in higher dimensions. Recall that a subgroup $G \leq \Homeo(X)$ is called \df{locally transitive} if for every open $U \sub X$ and every $x \in U$, $G_U \cdot x$ contains a neighborhood of $x$. Here $G_U$ denotes the \df{rigid stabilizer} of $U$, i.e., the subgroup of all elements of $G$ that fix all points in   $X \sminus U$. For example, if $X$ is a closed manifold, its full homeomorphism group is locally transitive. It was shown by Gutman \cite{Gutman2008} that if $X$ is a closed manifold of dimension at least $2$ or the Hilbert cube and $G$ acts locally transitively on $X$, then the action $G \actson \cC(X)$ is minimal but not $1$-transitive. In the same article \cite{Gutman2008}*{Question 12.3} it was asked if the UMF of the homeomorphism group of a  closed manifold of dimension $3$ or higher or the Hilbert cube equals $\cC(X)$. Pestov~\cite{Pestov2006}*{Open Problem~6.4.13} remarked that $\cC(X)$ ``currently looks like a likely candidate to serve as the UMF'' in these cases.  In this paper, we analyze the flow $\cC(X)$ and using the results from \cite{BenYaacov2017} and \cite{Zucker2018}, we answer both Uspenskij's and Gutman's questions in the negative.

\begin{theorem}
  \label{th:i:dimension2}
  Let $X$ be a closed manifold of dimension at least $2$ or the Hilbert cube and let $G$ be a locally transitive subgroup of $\Homeo(X)$. Then the universal minimal flow of $G$ is not metrizable.
\end{theorem}

In dimension $3$ or higher, we prove the following stronger result.

\begin{theorem}
  \label{th:i:dimension3}
  Let $X$ be a closed manifold of dimension at least $3$ or the Hilbert cube. Then the flow $\Homeo(X) \actson \cC(X)$ has meager orbits. In particular, if $G \leq \Homeo(X)$ is locally transitive, the UMF of $G$ has meager orbits.
\end{theorem}

Some important examples of locally transitive groups of homeomorphisms of a closed manifold $X$ include the path-component of the identity in $\Homeo(X)$ and the diffeomorphism group of $X$ if $X$ has a smooth structure (see \cite{Gutman2008}*{Example~3.1}).

We do not know whether Theorem~\ref{th:i:dimension3} holds in dimension $2$.
\begin{question}
  \label{q:sphere}
  Are the orbits of the action $\Homeo(S^2) \actson \cC(S^2)$ meager? Are the orbits of the UMF of $\Homeo(S^2)$ meager?
\end{question}

The paper is organized as follows. In Section~\ref{sec:meager-orbits-higher}, we give proofs of the two theorems above. In Section~\ref{sec:comeager-set-chains}, we give a description of the generic maximal connected chain in the Hilbert cube and prove that a closely related flow, that of \df{the partial connected chains} does have a comeager orbit.

\subsection*{Acknowledgments} Research was partially supported by the NSF grant no.\ DMS 1803489, the ANR project AGRUME (ANR-17-CE40-0026) and the \emph{Investissements d'Avenir} program of Université de Lyon (ANR-16-IDEX-0005). Y. G. was partially supported by the National Science Centre (Poland) grant 2016/22/E/ST1/00448.

\section{Proofs of Theorems \ref{th:i:dimension2} and \ref{th:i:dimension3}}
\label{sec:meager-orbits-higher}

Let $X$ be a closed manifold of dimension $m\geq 2$ or the Hilbert cube (in this case we set $m = \aleph_0$) and let $d$ be a compatible metric. For any $Y\subseteq X$ and $\delta > 0$, let $B_\delta(Y) = \{x\in X: d(x, Y) < \epsilon\}$.

    Note that as open subsets of manifolds are locally path connected, connectedness implies path-connectedness, so we will use those two terms interchangeably.

Next we describe the space of \emph{connected maximal chains} of closed subsets of $X$. Given any compact space $Y$, one can form $V(Y)$, the space of non-empty, compact subsets of $Y$ endowed with the Vietoris topology. If $\mathcal{B}$ is a basis for the topology of $Y$, a basis for $V(Y)$ is given by $V(\mathcal{B})$, where the typical member of $V(\mathcal{B})$ has the form
$$O(U_0,...,U_{n-1}):= \set[\Big]{K\in V(Y): K\cap U_i\neq \emptyset \text{ for } i < n \text{ and }  K\subseteq \bigcup_{i< n} U_i}$$
for $U_i\in \mathcal{B}$. The space of maximal chains of $X$, denoted $\Phi(X)$, is a subspace of $V(V(X)) =: V^2(X)$, where $c\in \Phi(X)$ iff $c$ satisfies the following conditions:
\begin{itemize}
	\item
	The members of $c$ are linearly ordered by inclusion;
	\item
	$c$ is maximal with this property, i.e.\ if $K\in V(X)$ is such that $c\cup \{K\}$ is also linearly ordered by inclusion, then $K\in c$.
\end{itemize}
A maximal chain $c\in \Phi(X)$ is \emph{connected} iff $c$ is connected as a compact subspace of $V(X)$. Equivalently, $c$ is connected iff each member of $c$ is connected (Lemma 2.3 of \cite{Gutman2008}). We let
\begin{equation*}
  \cC(X) = \set{c \in \Phi(X) : c \text{ is connected}}
\end{equation*}
denote the space of such chains. As $X$ will be fixed throughout this section, we also write $M := \cC(X)$ for brevity. The space $M$ is compact and it was shown in \cite{Gutman2008}*{Theorem~6.5} that if $G \leq \Homeo(X)$ is locally transitive, then the natural action of $G$ on $M$ is minimal.

We will prove that the UMFs of the various groups $G$ that we consider are non-metrizable by studying the minimal, metrizable flow $G \actson M$. It was shown in \cite{BenYaacov2017} that if the UMF of $G$ is metrizable, then it (and therefore every $G$-flow) has a comeager orbit. Thus to prove non-metrizability of the UMF, it suffices to show that $M$ has meager orbits. We will do this in dimension at least $3$ using the following criterion due to Rosendal (see \cite{BenYaacov2017} for a proof). Below $G_\epsilon$ denotes the $\epsilon$-ball around $1_G$ for some fixed right-invariant, compatible metric\footnote{The existence of such a metric is guaranteed by the Birkhoff--Kakutani theorem.} on $G$.

\begin{fact}
	\label{Fact:MeagerOrbits}
	Let $G$ be a Polish group, and suppose $Y$ is a Polish, topologically transitive $G$-space. Then the following are equivalent:
	\begin{enumerate}
		\item
		All orbits in $Y$ are meager;
		\item
		There is $\eps > 0$ and a non-empty open $U \subseteq Y$ such that for any non-empty open $V \subseteq U$, there are non-empty open $W_0, W_1\subseteq V$ such that $G_\eps  W_0\cap W_1 = \emptyset$.
	\end{enumerate}
\end{fact}

In dimension $2$, we are unable to verify that $M$ has meager orbits and we resort to a more general criterion isolated in \cite{Zucker2018}. Recall that for any $G$-flow $Y$, $\tS_G(Y)$ denotes the \emph{maximal highly proximal extension} of $Y$ (see \cite{Zucker2018} for the definition). We note that the equivalent conditions of Fact~\ref{Fact:MeagerOrbits} imply the assumptions of Fact~\ref{Fact:HPNonmetrizable}.

\begin{fact}
	\label{Fact:HPNonmetrizable}
	Let $G$ be a second countable group, and suppose $Y$ is a $G$-flow with the property that there is $\eps > 0$ and a collection $\{U_n: n \in \N \}$ of non-empty open subsets of $Y$ such that the sets $\{G_\eps U_n: n \in \N \}$ are pairwise disjoint. Then $\tS_G(Y)$ is non-metrizable. If $Y$ is minimal, then so is $\tS_G(Y)$.
\end{fact}

We note that the criteria of both Fact~\ref{Fact:MeagerOrbits} and Fact~\ref{Fact:HPNonmetrizable} become easier to verify if $G$ is smaller, so from here on, we assume without loss of generality that $G = \Homeo(X)$ and we equip it with the right-invariant metric $d_G$ given by $d_G(g,h) = \sup \set{d(gx, hx): x\in X}$.

We will need some of the tools from \cite{Gutman2008} to analyze the space $M$. A chain $c\in M$ is called a \emph{ray} if there is a continuous injective map $\phi\colon [0,\infty)\to X$ such that $\im{\phi}$ is dense and such that $c = \{X\}\cup \{\phi([0,t]): t \geq 0\}$. By abuse of language, we will often confound $c$ and $\phi$. In \cite{Gutman2008}, it is proven that the rays are dense in $M$. This is then used to provide a particularly useful $\pi$-base for the topology of $M$ (where a \emph{$\pi$-base} for a topological space $Y$ is a collection $\mathcal{B}$ of non-empty, open subsets of $Y$ such that every non-empty, open $U\subseteq Y$ contains a member of $\mathcal{B}$).

\begin{defn}
	\label{Def:Tube}
A \emph{tube} is a sequence of non-empty, open subsets $U_0,...,U_{n-1}\subseteq X$ satisfying the following properties:
\begin{itemize}
	\item
      Each $U_i$ is connected;
	\item
	For $i < j < n$, we have $\overline{U_i}\cap \overline{U_j}\neq \emptyset$ iff $U_i \cap U_j \neq \emptyset$ iff $j = i+1$.
  \end{itemize}
\end{defn}

We will refer to the sets $U_i$ as the \df{links} of the tube and denote by $\bigcup \cU$ the union of all links of $\cU$. A \emph{subtube} of $\cU = \langle U_0,...,U_{n-1}\rangle$ is a tube of the form $\langle U_k,...,U_\ell\rangle$ for some $0\leq k \leq \ell< n$. If $k< n$ and we set $\cV := \langle U_0,...,U_{k-1}\rangle$ and $\cW := \langle U_k,...,U_{n-1}\rangle$, we sometimes write $\cU = \cV^\frown \cW$.

If $\mathcal{U} := \langle U_0,...,U_{n-1}\rangle$ is a tube, we can associate to it the following open subset of $M$:
$$O(\mathcal{U}) := \{c\in M: c\cap O(U_0,...,U_k)\neq \emptyset \text{ for each } k < n\}.$$
We will also refer to $O(\mathcal{U})$ as a tube. It is shown in \cite{Gutman2008} that every ray has a neighborhood basis consisting of tubes. It follows that the collection of tubes forms a $\pi$-base for $M$.

\begin{figure}
  \includegraphics[height=8cm]{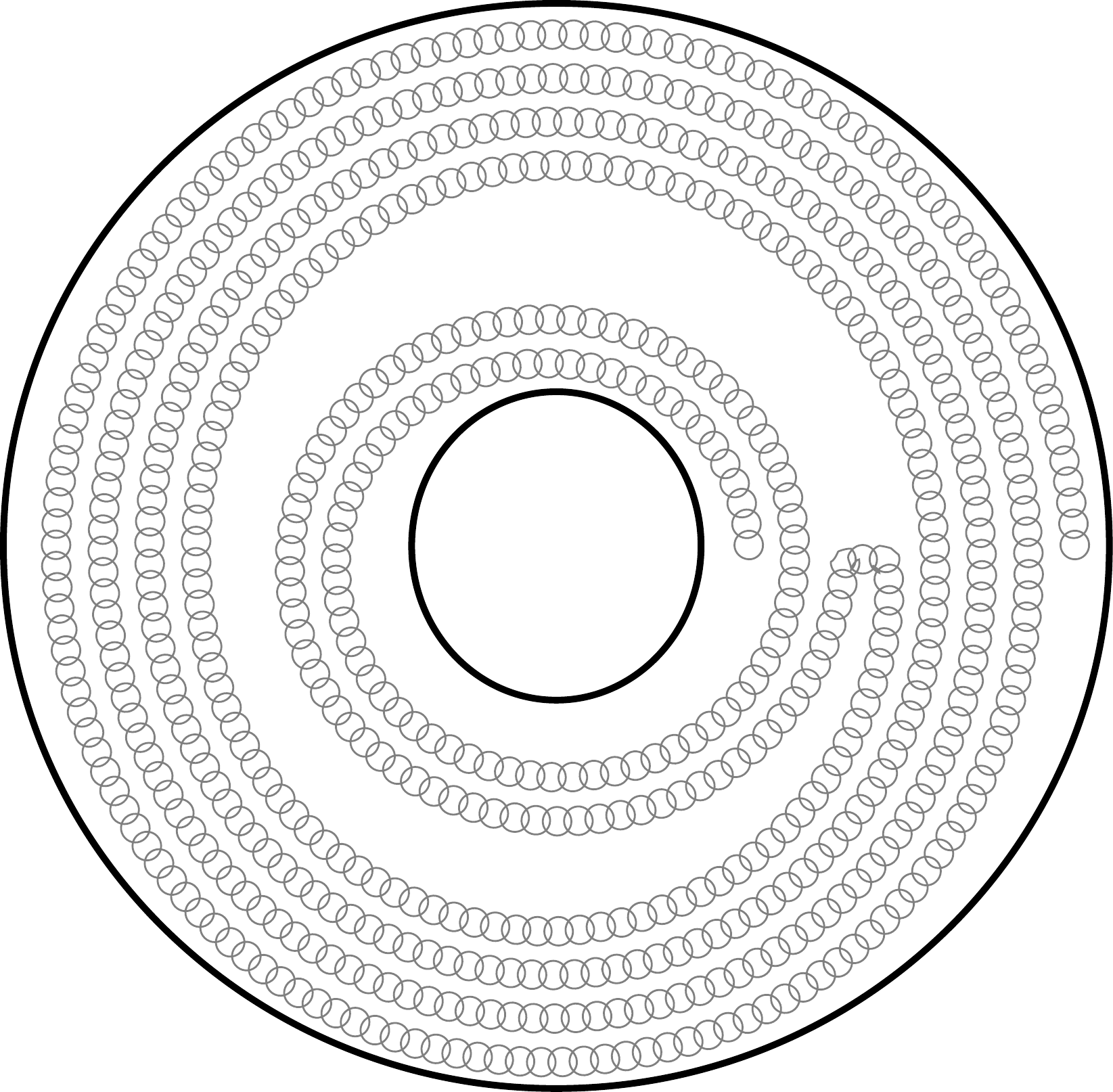}
  \caption{The tube $\cU_2$}
  \label{fig:tube-U2}
\end{figure}

Before proceeding with more technical details, we describe the strategy of the proofs. We start with the weaker result in dimension $2$, as it is simpler and can serve as a guide for reading the rest of the paper. Our argument is entirely local and does not depend on the global structure of the manifold: all constructions are done in a fixed closed disc $Z \sub X$.  First we fix an annulus $A$ in $Z$. The main idea is to consider tubes $\set{\cU_N : N \in \N}$ that wind around the annulus, first counterclockwise $N$ times, then clockwise $2N$ times (see Figure~\ref{fig:tube-U2}). We show that there exists $\eps > 0$ such that if $N$ and $N'$ are sufficiently far apart, then $G_\eps \cdot O(\cU_N) \cap G_\eps \cdot O(\cU_{N'}) = \emptyset$, thus verifying the assumptions of Fact~\ref{Fact:HPNonmetrizable}. It is harder to witness the criterion of Fact~\ref{Fact:MeagerOrbits} because the adversary gives us the beginning of the two tubes (the open set $V \sub M$) and we only control the tails. The reason we need an extra dimension in this case is to be able to avoid the part of the chain that has already been constructed while we are winding around.

Now we return to the technical definitions. Instead of working with rays, it is often easier to work with \emph{arcs}, i.e.\ continuous injections from an interval $[s,t]$ to $X$. We often simply take $s = 0$ and $t = 1$. We will use the term \emph{path} to refer to a continuous function from an interval $[s,t]$ to $X$ which may not be injective.

\begin{defn}
	\label{Def:Compatible}
If $\cU = \langle U_0, \ldots, U_{n-1} \rangle$ is a tube and $\phi\colon [0, 1]\to X$ is an arc, we will say that $\phi$ and $\cU$ are \df{compatible} (or that \df{$O(\cU)$ contains $\phi$}) if $\phi(0) \in U_0$, $\phi(1) \in U_{n-1}$ and $\im \phi \sub \bigcup \cU$.
\end{defn}

Note that for every tube, there is an arc compatible with it. Also note that if $\phi$ and $\cU$ are compatible, this implies that any ray extending $\phi$ belongs to $O(\cU)$.

\begin{defn}
	\label{Def:Refines}
If $\cU = \langle U_0, \ldots, U_{n-1}\rangle$ and $\cV = \langle V_0, \ldots, V_{N-1}\rangle$ are tubes, we will say that $\cV$ \df{refines} $\cU$ if $V_0 \sub U_0$, and for all $i$ there is a $j$ such that $V_i \sub U_j$, and for all $j$ there is an $i$ such that $V_i \sub U_j$.
\end{defn}

Note that if $\cV$ refines $\cU$, then $O(\cV) \sub O(\cU)$. We note that if $\phi$ and $\cV$ are compatible, then $\phi$ and $\cU$ are compatible if the last link of $\cV$ is contained in the last link of $\cU$.

\textbf{We now fix once and for all} a closed Euclidean neighborhood $Z'$ in $X$, which we identify with $[-2,2]^m$. If $x \in Z'$, we will denote by $(x_0, \ldots, x_{m-1})$ its coordinates coming from the identification of $Z'$ with $[-2, 2]^m$. We may assume that the metric $d$ on $X$ restricts to the usual Euclidean metric on $Z'$. In the case of the Hilbert cube, we let $Z'$ be the entire space and we use the metric
\begin{equation*}
  d(x, y) = \sum_n 2^{-n}|x_n - y_n|.
\end{equation*}

If $x, y \in Z'$, denote by $[x, y] $ the line segment between $x$ and $y$. An arc in $Z'$ is called \df{piecewise linear} if it is a concatenation of finitely many line segments. We note the following basic lemma.
\begin{lemma}
  \label{l:piece-linear-path}
  Let $U \sub Z'$ be open and connected and let $x, y \in U$. Then there exists a piecewise linear arc $\phi \colon [0, 1] \to U$ such that $\phi(0) = x$, $\phi(1) = y$.
\end{lemma}
\begin{proof}
  Note that the line segment between the centers of two intersecting balls in $\R^m$ is contained in their union.
  Now let $\cB$ be the collection of all open balls contained in $U$. Form a graph on $\cB$ by connecting two balls $B_1, B_2 \in \cB$ with an edge if $B_1 \cap B_2 \neq \emptyset$ and note that as $U$ is connected, this graph is connected. Let $C, D \in \cB$ be balls with centers $x$ and $y$ respectively and let $C = B_0, \ldots, B_{n-1} = D$ be a path in $\cB$ connecting $C$ to $D$. Let $z_i$ be the center of $B_i$. Then $\bigcup_i [z_i, z_{i+1}]$ is a piecewise linear path contained in $U$ connecting $x$ to $y$. One can make it injective by erasing all loops.
\end{proof}

If $m < \aleph_0$, let $Z = [-1, 1]^m \sub Z'$; if $X$ is the Hilbert cube, let $Z = Z' = X$. In most of what follows, we will only work in $Z$. ($Z'$ is only needed in the proof of Lemma~\ref{l:pw-linear-Z}.) In particular, we will need arcs and tubes which behave nicely with respect to $Z$.

\begin{defn}
	\label{Def:SimplyConnected}
	\begin{enumerate}
		\item
If $\cU = \langle U_0, \ldots, U_{n-1} \rangle$ is a tube, we will say that $\cU$ is \df{simply connected} if each $U_i$ is simply connected and for all $i < n-1$, $U_i \cap U_{i+1}$ is connected. Note that by Van Kampen's theorem, this implies that for all $k < \ell$, the union $U_k \cupdots U_\ell$ is simply connected.
		\item
Let $Y \sub X$. We will say that $\cU$ is \df{simply connected in $Y$} if every subtube $\cU' = \langle U_k, \ldots, U_\ell \rangle$ of $\cU$ which is entirely contained in $Y$ is simply connected.
	\end{enumerate}
\end{defn}
The idea is the following: if we could work entirely with simply connected tubes, the proof would simplify. Simply connected tubes are easy to create in a given Euclidean region; however, we lose control of this when moving from one chart to another. Therefore, we must content ourselves to only demanding that tubes be simply connected in one Euclidean region. In the case of the Hilbert cube, there are no such difficulties and we can take $Z = Z'$ to be the entire space, which simplifies some aspects of the proof.

Say that an arc $\phi \colon [0, 1] \to X$ is \df{piecewise linear in $Z$} if for every open interval $(a, b) \sub [0, 1]$ such that $\phi((a, b)) \sub \Int Z$, we have that $\phi|_{[a, b]}$ is piecewise linear.

\begin{lemma}
  \label{l:pw-linear-Z}
  \begin{enumerate}
  \item \label{i:pwlZ:1} For every tube, there is an arc compatible with it which is piecewise linear in $Z$.

   \item \label{i:pwlZ:2} Let $\cU$ be a tube such that $\bigcup \cU \sub Z$ and let $\phi$ be a piecewise linear arc compatible with $\cU$. Then there exists a simply connected tube $\cV$ refining $\cU$ such that $\phi$ is compatible with $\cV$.

  \item \label{i:pwlZ:3} For every tube $\cU$ and arc $\phi$ compatible with $\cU$ and piecewise linear in $Z$, there exists a tube $\cV$ refining $\cU$ and compatible with $\phi$ which is simply connected in $Z$.

  \item \label{i:pwlZ:4} The collection of tubes which are simply connected in $Z$ forms a $\pi$-basis for $M$.
  \end{enumerate}
\end{lemma}
\begin{proof}
  \ref{i:pwlZ:1} Let $\cU = \langle U_0, \ldots, U_{n-1} \rangle$ be a tube, and let $\phi\colon [0,1]\to X$ be a compatible arc. Find $0 = s_0 < s_1 <\cdots < s_N = 1$ such that $\phi([s_j, s_{j+1}])$ has diameter at most $1/10$ and is contained in some $U_i$ for each $j< N$. We specifically demand that $\phi([s_0, s_1])\subseteq U_0$ and $\phi([s_{N-1}, s_N])\subseteq U_{n-1}$. Choose $\delta > 0$ small enough so that by setting $V_j = B_\delta(\phi([s_j, s_{j+1}]))$, we obtain a tube $\cV := \langle V_0,...,V_{N-1}\rangle$ which refines $\cU$, with $V_{N-1}\subseteq U_{n-1}$, and with the property that any link of $\cV$ which meets $Z$ is contained in $Z'$.

  Let $\set{[k_i, \ell_i] : i < q}$ be maximal subintervals of $[0, N-1]$ such that for each $i$, $\bigcup_{j = k_i}^{\ell_i} V_j \sub Z'$. We assume that $0 < k_0$ and $\ell_{q-1} < N-1$; if this is not the case, the construction can be modified accordingly. Notice that for each $i<q$, we have that $V_{k_i}\cap V_{k_i-1}\cap Z = V_{\ell_i}\cap V_{\ell_i+1}\cap Z = \emptyset$. For each $i<q$, let $z_i \in V_{k_i} \cap V_{k_i-1}$ and $z_i' \in V_{\ell_i} \cap V_{\ell_i+1}$. Let $z_{-1}' \in V_0 \sminus Z$, $z_q  \in V_{N-1} \sminus Z$. For each $i$, let $\psi_i$ be an arc from $z_{i-1}'$ to $z_i$ contained in $\bigcup_{j = \ell_{i-1}+1}^{k_i-1} V_j$ and, using Lemma~\ref{l:piece-linear-path}, let $\psi_i'$ be a piecewise linear arc from $z_i$ to $z_i'$ contained in $\bigcup_{j = k_i}^{\ell_i} V_j$. Finally, concatenate all of those and erase the loops.

    \ref{i:pwlZ:2} Write $\cU = \langle U_0, \ldots, U_{n-1} \rangle$. Let $\delta > 0$ and $0 = t_0 < \cdots < t_k = 1$ be such that, denoting $V_i = B_\delta\big([\phi(t_i), \phi(t_{i+1})]\big)$:
  \begin{itemize}
  \item $\phi$ is a line segment between $t_i$ and $t_{i+1}$;
  \item $\cV = \langle V_0, \ldots, V_{k-1} \rangle$ is a tube;
  \item $\cV$ refines $\cU$.
  \end{itemize}
  The tube $\cV$ is simply connected because each $V_i$ is convex.

  \ref{i:pwlZ:3} Let $\cU = \langle U_0, \ldots, U_{n-1}\rangle$ and $\phi \colon [0, 1] \to X$ be given.
  Let $\cU' = \langle U_q, \ldots, U_r \rangle$ be a maximal subtube of $\cU$ contained entirely in $Z$. Let $(a, b) \sub [0, 1]$ be an interval with $\phi((a, b)) \sub \bigcup\cU'$, $\phi(a)\in U_{q-1}\cap U_q$, and $\phi(b)\in U_r\cap U_{r+1}$. By the assumptions on $\phi$, $\phi|_{[a, b]}$ is piecewise linear.  Using \ref{i:pwlZ:2}, let $\cV = \langle V_0, \ldots, V_{k-1} \rangle$ be a simply connected tube refining $\cU'$ compatible with $\phi|_{[a, b]}$. Let $V_\ell$ be the last element of $\cV$ that intersects $U_{q-1}$ and let $V_p$ be the first element that intersects $U_{r+1}$. Then $\langle U_0, \ldots, U_{q-1}, V_{\ell}, \ldots, V_{p}, U_{r+1}, \ldots, U_{n-1} \rangle$ is a tube that refines $\cU$. Then one can repeat this procedure in order to replace all maximal subtubes of $\cU$ contained in $Z$ with simply connected ones.

  \ref{i:pwlZ:4} Follows from \ref{i:pwlZ:2} and \ref{i:pwlZ:3}.
\end{proof}

We define three closed \emph{annuli} in $Z$ as follows:
\begin{align*}
B :=& \{x \in Z : 5/11 \leq x_0^2 + x_1^2 \leq 6/10\},\\
A :=& \{x \in Z : 1/2 \leq x_0^2 + x_1^2 \leq 3/4\},\\
A':=& \{x \in Z : 1/4 \leq x_0^2 + x_1^2 \leq 9/10\}.
\end{align*}
(They are real annuli only if $m = 2$; otherwise they are rather the pullbacks of annuli by the projection map. We still call them annuli even if $m > 2$ by abuse of language.) Note that $B \sub A \sub A'$. The common \emph{center} of the annuli is the set
\begin{equation*}
  \{(0,0,x_2,...,x_{m-1}) : -1 \leq x_2,...,x_{m-1} \leq 1\}.
\end{equation*}
We also define the function $\alpha \colon A' \to \bT$ by
\begin{equation*}
  \alpha(x) = (x_0 + i x_1)/(x_0^2 + x_1^2)^{1/2}.
\end{equation*}
(Here $\bT = \set{z \in \C : |z| = 1}$ is the unit circle and $\alpha(x)$ is just the argument of $x_0 + i x_1$.)

Now suppose $\phi\colon [a, b]\to A'$ is a path. Then there is a unique continuous lifting $\omega_\phi \colon [a, b] \to \R$ of the map $\alpha \circ \phi$ such that $\omega_\phi(a) \in [0, 1)$ and $p \circ \omega_\phi = \alpha \circ \phi$, where $p \colon \R \to \bT$, $p(t) = \exp(2 \pi i t)$ is the standard covering map (see, for example, \cite{Munkres2000}*{Lemma~54.1}).

The \emph{winding number} of $\phi$ is then given by $w(\phi) := \omega_\phi(b) - \omega_\phi(a)$. Note that the winding number of a path does not depend on its parametrization. If $\phi$ and $\psi$ are paths with domain $[0,1]$ whose images are contained in $A'$ such that $\psi(1) = \phi(0)$, we denote by $\phi \cdot \psi$ their concatenation and by $\phi^{-1}$ the inverse path given by $\phi^{-1}(t) = \phi(1-t)$. We note the following basic properties of the winding number that will be used repeatedly.

\begin{lemma}
  \label{l:basic-winding}
  Let $\phi$ and $\psi$ be paths $[0, 1] \to A'$ with $\psi(1) = \phi(0)$. Then the following hold:
  \begin{enumerate}
  \item \label{i:basicW:1} $w(\phi \cdot \psi) = w(\phi) + w(\psi)$;
  \item \label{i:basicW:2} $w(\phi^{-1}) = -w(\phi)$;
  \item \label{i:basicW:3} If $\im \phi \sub U$ with $U \sub A'$ simply connected and $\phi(0) = \phi(1)$, then $w(\phi) = 0$.
  \end{enumerate}
\end{lemma}
\begin{proof}
  \ref{i:basicW:1} and \ref{i:basicW:2} are clear from the definition. For \ref{i:basicW:3}, let $(\phi_t : t \in [0, 1])$ be a homotopy of closed paths in $U$ from $\phi$ to the constant path $s \mapsto \phi(0)$. The function $t \mapsto w(\phi_t)$ is continuous and it takes only integer values, so it must be constant. This implies that $w(\phi) = 0$.
\end{proof}

\textbf{We now fix once and for all} a suitably small $\epsilon > 0$. Taking $\epsilon = 1/1000$ will suffice. Recall that $G_\eps = \set{g \in G : d_G(g, 1_G) < \eps}$. The next lemma says that the winding number is stable under perturbations smaller than $\eps$.
\begin{lemma}
	\label{l:PerturbWinding}
	Let $\phi\colon [0,1]\to A'$ be a path. Let $g\in G_\epsilon$ be such that $g\cdot \im{\phi}\subseteq A'$. Then for any $t\in [0,1]$, we have $|\omega_{g\phi}(t) - \omega_\phi(t)| < 3/2$; in particular, we have that $|w(\phi) - w(g\cdot\phi)| < 3$. 
\end{lemma}
\begin{proof}
  Suppose not. As $|\omega_{g\phi}(0) - \omega_\phi(0)| < 1$ and $|\omega_{g\phi}(t) - \omega_\phi(t)|$ is continuous, there is $t_0 \in [0, 1]$ such that $|\omega_{g\phi}(t_0) - \omega_\phi(t_0)| = 3/2$, which implies that $\alpha(\phi(t_0)) = -\alpha(g \cdot \phi(t_0))$. This contradicts the fact that $d(\phi(t_0), g \cdot \phi(t_0)) < \eps$ and that $\phi(t_0), g \cdot \phi(t_0) \in A'$.
\end{proof}

\begin{defn}
	\label{Def:Confined}
Let $\mathcal{U} := \langle U_0,...,U_{n-1}\rangle$ be a tube which is simply connected in $A'$. We say that an arc $\phi\colon [0,1]\to X$ is \emph{$(\cU, A')$-confined} if $\im{\phi}\subseteq (A'\cap \bigcup \mathcal{U})$ and
for any $k < n$ with $U_k\cap \im{\phi}\neq \emptyset$, we have $U_k\subseteq A'$.
\end{defn}

The next lemma shows that the winding number of an arc is roughly determined by any sufficiently fine tube that contains it.
\begin{lemma}
  \label{l:WindingBound}
  Suppose $\cU = \langle U_0,...,U_{n-1}\rangle$ is a tube which is simply connected in $A'$. Let $\phi$ and $\psi$ be arcs which are $(\cU, A')$-confined. Suppose further that each $U_k$ has diameter at most $1/10$, and that for some $i, j< n$ we have $\phi(0), \psi(0) \in U_i$, $\phi(1), \psi(1) \in U_j$. Then the following hold:
    \begin{enumerate}
    \item \label{i:l:WB:1} $|w(\phi)| \leq |i-j| + 1$;
    \item \label{i:l:WB:2} $|w(\phi) - w(\psi)| \leq 1$.
    \end{enumerate}
\end{lemma}
\begin{proof}
  \ref{i:l:WB:1} We may assume that $\im{\phi}$ meets each $U_k$, so $U_k\subseteq A'$ for each $k< n$ and $\cU$ is simply connected. Without loss of generality, we may assume that $r = |i - j|$ is the least number $|i' - j'|$ such that $\phi(0) \in U_{i'}$ and $\phi(1) \in U_{j'}$. We proceed by induction on $r$. For $r = 0$, we have that $\phi(0), \phi(1)\in U_i$. Let $\theta\colon [0,1]\to U_i$ be a path with $\theta(0) = \phi(1)$ and $\phi(1) = \theta(0)$. Then since $\bigcup \cU \subseteq A'$ is simply connected, we have $w(\phi \cdot \theta) = 0$. Notice that $|w(\theta)| \leq 1$ since $\im{\theta}\subseteq U_i$ and $\alpha(U_i)$ has small diameter. Therefore we also have $|w(\phi)|\leq 1$ as desired.

  Now suppose the result is true for $r = k$, and suppose $\phi$ is a path with $r = k+1$. Without loss of generality, assume that $i < j$ and note that by the choice of $r$, $\phi(0)\not\in U_{i+1}$. Let $t = \sup(\phi^{-1}(U_i))$. We must have $\phi(t)\in U_{i+1}$, so we can find an interval $(a,b)\subseteq \phi^{-1}(U_{i+1})$ with $t\in (a,b)$. Then find $s\in (a,b)$ with $\phi(s)\in U_i$. By the induction hypothesis, we have that $|w(\phi|_{[0,s]})| \leq  1$ and $|w(\phi|_{[s,1]})| \leq k+1$. Hence $|w(\phi)|\leq k + 2$ as desired.
	
\ref{i:l:WB:2} We may assume that each $U_k$ meets $\im{\phi}$ or $\im \psi$, so $U_k\subseteq A'$ for each $k< n$ and $\cU$ is simply connected. Let $\theta_0$ and $\theta_1$ be arcs with $\mathrm{Im}(\theta_0)\subseteq U_i$, $\im{\theta_1}\subseteq U_j$,  $\theta_0(0) = \psi(0)$, $\theta_0(1) = \phi(0)$, $\theta_1(0) = \phi(1)$ and $\theta_1(1) = \psi(1)$. Then since $\bigcup \mathcal{U}\subseteq A'$ is simply connected, we have $w(\theta_0 \cdot \phi \cdot \theta_1 \cdot \psi^{-1}) = 0$. Since $|w(\theta_i)|< 1/2$ for $i = 0, 1$, the result follows.
\end{proof}

We are now ready to prove our main results.
\begin{proof}[Proof of Theorem~\ref{th:i:dimension2}]
As already noted, when $m \geq 3$, Theorem~\ref{th:i:dimension2} follows from Theorem~\ref{th:i:dimension3}. So now we will assume that $m = 2$ and we will aim to apply Fact~\ref{Fact:HPNonmetrizable}. This case will be simpler than proving meager orbits for dimension $m\geq 3$, and we will not need all of the machinery we have developed so far. We will also be much more explicit with the various arcs and tubes we create.

First for $t\in [0,1]$, we set $r(t) = 9/12 + t/2$. Notice that for $t\in [0,1]$, we have $1/\sqrt{2} < r(t) < \sqrt{3}/2$, so $r(t)$ is a radius between the inner and outer radius of the annulus $A$. For each $N\in \N \setminus \set{0}$, define the arc $\phi_N\colon [0,1]\to A$, using polar coordinates, as follows:

\begin{align*}
	\phi_N(t) =
\begin{cases}
r(t) \exp(6\pi iNt) \quad &\text{if } 0 \leq t \leq 1/3,\\[1mm]
r(t) \quad &\text{if } 1/3\leq t \leq 2/3,\\[1mm]
r(t) \exp(-12\pi iNt)\quad &\text{if } 2/3 \leq t \leq 1.
\end{cases}
\end{align*}
Thus $\phi_N$ winds counterclockwise $N$ times, then clockwise $2N$ times, while slowly expanding in radius. More explicitly, the winding function of $\phi_N$ is
\begin{align*}
w_{\phi_N}(t) =
\begin{cases}
3Nt \quad &\text{if } 0 \leq t \leq 1/3,\\[1mm]
N \quad &\text{if } 1/3\leq t \leq 2/3,\\[1mm]
N-6N(t-2/3)\quad &\text{if } 2/3 \leq t \leq 1.
\end{cases}
\end{align*}
In particular, we have $w(\phi_N) = -N$.

Let $\mathcal{U}_N = \langle (U_N)_0,...,(U_N)_{j_N-1}\rangle$ be a simply connected tube compatible with $\phi_N$ such that each $(U_N)_j$ has diameter at most $1/100$ and $\bigcup \cU_N\subseteq A$. (See Figure~\ref{fig:tube-U2} on p.~\pageref{fig:tube-U2}.)

For the following lemma, recall that we have fixed $\epsilon = 1/1000$.

\begin{lemma}
	\label{Lem:Winding2D}
	Suppose $\psi$ is a ray contained in $O(\mathcal{U}_N)$, and let $g\in G_\epsilon$. Let
\begin{equation*}
  s = \max\{t\geq 0: (g\cdot \psi)([0,t])\subseteq A'\}
\end{equation*}
   and let $N\geq 100$. Then the following hold:
	\begin{enumerate}
		\item
		\label{i:l:2D:1}
		$\omega_{g\psi}\colon [0,s]\to \R$ does not reach $3N/2$ before reaching $-N/2$;
		\item
		\label{i:l:2D:2}
		$\omega_{g\psi}\colon [0,s]\to \R$ reaches $3N/4$ before reaching $-N/4$.
	\end{enumerate}
\end{lemma}

\begin{proof}
	For \ref{i:l:2D:1}, let $q\leq s$ be such that $\psi|_{[0,q]}$ is compatible with $\mathcal{U}_N$. Then by Lemma~\ref{l:WindingBound}~\ref{i:l:WB:2}, we have $|w(\psi|_{[0,q]}) - w(\phi_N)|\leq 1$. In particular, we have $\omega_\psi(q) \leq 2-N < -3N/4$. Towards a contradiction, suppose that for some $p\leq q$ we have $\omega_\psi(p) \geq 5N/4$. Find $a\in [0,1]$ and  $j < j_N$ such that $\phi_N(a), \psi(p)\in (U_N)_j$. Applying Lemma~\ref{l:WindingBound}~\ref{i:l:WB:2} to $\phi_N|_{[0, a]}$ and $\psi|_{[0, p]}$, we see that this is not possible. To transfer this observation about $\psi$ into one about $g\psi$, we note that $\im{\psi}\subseteq A'$ and $\im{g\psi}\subseteq A'$, letting us apply Lemma~\ref{l:PerturbWinding} to conclude that $|\omega_{g\psi}(t) - \omega_\psi(t)|< 3/2$ for any $t\in [0,1]$. So $\omega_{g\psi}(q)\leq 4-N < -3N/4$, and for any $p\leq q$, we have $\omega_{g\psi}(p)< 5N/4+3/2 < 3N/2$. 
	
	For \ref{i:l:2D:2}, let $\mathcal{V}_N$ be an initial segment of $\mathcal{U}_N$ which is compatible with $\phi_N|_{[0,1/3]}$. Let $q\leq s$ be such that $\psi|_{[0,q]}$ is compatible with $\cV_N$. Then by Lemma~\ref{l:WindingBound}~\ref{i:l:WB:2}, we must have $w_\psi(q)\geq N-1 > 7N/8$. Showing that we do not have $w_\psi(p) \leq -3N/8$ for any $p\leq q$ is done as in the proof above, and then we apply Lemma~\ref{l:PerturbWinding} as before.
\end{proof}

Using Lemma~\ref{Lem:Winding2D}, we see that if we set $U_n := O(\cU_{100\cdot 2^n})\subseteq M$, then $U_n$ is non-empty, open and the sets $\{G_\epsilon\cdot U_n: n \in \N\}$ are pairwise disjoint. We now apply Fact~\ref{Fact:HPNonmetrizable} to conclude the proof of Theorem~\ref{th:i:dimension2} in the case $m = 2$.
\end{proof}  

Recall that if $m \geq 3$, then $X$ is \emph{strongly $\R$-inseparable (SRI)}, i.e., for any non-empty, open and connected $U\subseteq X$ and any continuous injection\footnote{Note that the injection is into $X$, not $U$.} $\phi\colon [0,1]\to X$, the set $U\setminus \im{\phi}$ is connected and non-empty (see \cite{Gutman2008}*{Theorem~A.3}).

\begin{proof}[Proof of Theorem~\ref{th:i:dimension3}]
Now we consider the case where $m\geq 3$ and we aim to apply Fact~\ref{Fact:MeagerOrbits} with $U = M$. So let $V\subseteq M$ be non-empty, open. Find some ray $c\in V$; by reparametrizing, we may assume that $c|_{[0,1]}$ is an arc such that every ray extending $c|_{[0,1]}$ belongs to $V$ and such that $c(1)\in \mathrm{Int}(B)$, where $B$ is the smallest annulus defined after Lemma~\ref{l:pw-linear-Z}. Let $\cU$ be any tube compatible with $c|_{[0,1]}$ with $O(\cU)\subset V$. By Lemma~\ref{l:pw-linear-Z}~\ref{i:pwlZ:1}, we find $\phi\colon [0,1]\to X$ an arc compatible with $\cU$ and piecewise linear in $Z$, and by Lemma~\ref{l:pw-linear-Z}~\ref{i:pwlZ:3}, we may refine $\cU$ to some $\cV := \langle V_0,...,V_{n-1}\rangle$ compatible with $\phi$ so that $\mathcal{V}$ is simply connected in $Z$, each $V_i$ is of diameter at most $\epsilon$, and $V_{n-1}\subseteq B$.

We will create two smaller neighborhoods contained in $O(\mathcal{V})$ by extending the arc $\phi$ in two different ways to arcs $\phi_0, \phi_1\colon [0,N]\to X$, where we have $\phi_i|_{[0,t]} = \phi|_{[0,t]}$ for some $t < 1$ with $\phi(t)\in V_{n-1}$ and $N = 100 n$. We first divide the annulus $B$ into four overlapping regions:
\begin{align*}
B_0 =& \{x \in B: x_0 \geq 0\},\\
B_1 =& \{x \in B: x_1 \geq 0\},\\
B_2 =& \{x \in B: x_0 \leq 0\},\\
B_3 =& \{x \in B: x_1\leq 0\}.
\end{align*}

We may assume that $V_{n-1}\subseteq B_0\cap B_1$. First we focus on constructing $\phi_0$. Fix a list of distinct points $y^0_1,...,y^0_N\in B\setminus \phi([0,1))$ with $y^0_1 = \phi(1)$ such that $y^0_i \in \mathrm{Int}(B_{i-1}\cap B_i)$, where the index $i$ in $B_i$ is interpreted mod $4$. The path $\phi_0$ will follow $\phi$ and then essentially visit all points $y^0_i$ in order. However, because we need $\phi_0$ to be injective, some care is needed.

Our construction of $\phi_0$ will proceed by building auxiliary arcs $\psi^0_i\colon [0,i]\to X$ for $1\leq i\leq N$. We will require that $\psi^0_i(i) = y^0_i$ for all $i$. To start, set $\psi^0_1 = \phi$. Next suppose that $\psi^0_k$ has been constructed. Find some $s^0_k$ with $k-1 < s^0_k < k$ such that $\psi^0_k([s^0_k, k])\subseteq \mathrm{Int}(B_{k-1}\cap B_k)$. If $k = 1$, we additionally demand that $\psi^0_k([s^0_k, k])\subseteq V_{n-1}$. Since $X$ is SRI,
$$A^0_k:= \mathrm{Int}(B_k)\setminus \left(\psi^0_k([0, s^0_k])\cup \{y^0_{k+2},...,y^0_N\}\right)$$
is connected, and by induction we will have that $y^0_k, y^0_{k+1}\in A^0_k$. Let $\theta^0_k\colon [0,1]\to A^0_k$ be a piecewise linear arc with $\theta^0_k(0) = y^0_k$ and $\theta^0_k(1) = y^0_{k+1}$. Let $t^0_k > s^0_k$ be minimal such that for some $a^0_k\in [0,1]$, we have $\psi^0_k(t^0_k) = \theta^0_k(a^0_k)$. The arc $\psi^0_{k+1}$ will be the concatenation $\psi^0_k|_{[0,t^0_k]} \cdot \theta^0_k|_{[a^0_k, 1]}$, where we parametrize as follows:
\begin{itemize}
	\item
	$\psi^0_{k+1}|_{[0,k-1]} = \psi^0_k|_{[0,k-1]}$;
	\item
	$\psi^0_{k+1}|_{[k-1,k]}$ is a reparametrization of $\psi^0_k|_{[k-1, t^0_k]}$;
	\item
	$\psi^0_{k+1}|_{[k, k+1]}$ is a reparametrization of $\theta^0_k|_{[a^0_k, 1]}$.
\end{itemize}
We conclude by setting $\phi_0 = \psi^0_N$ and note the properties of $\phi_0$ that we will need:
\begin{itemize}
	\item
	$\phi_0|_{[0,1]}$ is a reparametrization of $\phi|_{[0,t^0_1]}$ and is contained in the tube $\mathcal{V}$, with $\phi_0(0) \in V_0$ and $\phi_0(1)\in V_{n-1}$;
	\item
	$\phi_0(i)\in \mathrm{Int}(B_i\cap B_{i-1})$ for $1\leq i\leq N$;
	\item
	$\phi_0([i, i+1])\subseteq \mathrm{Int}(B_i)$ for $1\leq i\leq N-1$.
\end{itemize}
The arc $\phi_1$ is formed in a similar fashion, except that we travel the other way around the annulus. More precisely, we fix a list of distinct points $y^1_1,...,y^1_N\in B\setminus \phi([0,1))$ with $y^1_1 = \phi(1)$ such that $y^1_i\in \mathrm{Int}(B_{2-i}\cap B_{1-i})$. The procedure for building the auxiliary arcs $\psi^1_i$ is nearly identical, and we omit the details. We set $\phi_1 = \psi^1_N$, and note the following properties of $\phi_1$:
\begin{itemize}
	\item
	$\phi_1|_{[0,1]}$ is a reparametrization of $\phi|_{[0, t^1_1]}$ and is contained in the tube $\mathcal{V}$, with $\phi_1(0) \in V_0$ and $\phi_1(1)\in V_{n-1}$;
	\item
	$\phi_1(i)\in \mathrm{Int}(B_{2-i}\cap B_{1-i})$ for $1\leq i\leq N$;
	\item
	$\phi_1([i, i+1])\subseteq \mathrm{Int}(B_{1-i})$ for $1\leq i\leq N-1$.
  \end{itemize}
  We also note that both $\phi_0$ and $\phi_1$ are piecewise linear in $Z$.

Next we construct two tubes $\mathcal{T}_0$ and $\mathcal{T}_1$ compatible with the arcs $\phi_0$ and $\phi_1$ for which we will show that $O(\cT_0) \cap G_\eps \cdot O(\cT_1) = \emptyset$. The construction is very similar to the constructions in Lemma~\ref{l:pw-linear-Z} \ref{i:pwlZ:1} and \ref{i:pwlZ:2}. We can arrange so that the open sets appearing in each $\mathcal{T}_i$ have diameter less than $\epsilon$. Furthermore, since $\phi_0$ and $\phi_1$ share an initial segment, we can arrange that $\mathcal{T}_i = \mathcal{W}^\frown \mathcal{R}_i$, where $\mathcal{W} = \langle W_0,...,W_{r-1}\rangle$ refines $\cV$ and is compatible with $\phi|_{[0,t^0_1]}$, where without loss of generality we suppose that $t^0_1 \leq t^1_1$. We can also arrange that $W_{r-1}\subseteq V_{n-1}$. This will imply that $O(\mathcal{T}_i)\subseteq O(\mathcal{V})$ for $i = 0,1$. We also write $\mathcal{R}_0 = \langle
(R_0)_0,...,(R_0)_{r_0-1}\rangle$, $\mathcal{R}_1 = \langle (R_1)_0,...,(R_1)_{r_1-1}\rangle$ and require that $(R_i)_j\subseteq B$ for $i = 0,1$ and $j < r_i$. By choosing the $s^i_1$ close enough to $\phi(1)$, we can also arrange that $\phi([t^0_1, t^1_1])\subseteq (R_1)_0$. In particular, each $\phi_i|_{[1, N]}$ is an arc with $\im{\phi_i}\subseteq \bigcup \mathcal{R}_i$ and satisfies $\phi_i(1)\in (R_i)_0$ and $\phi_i(N)\in (R_i)_{r_i-1}$. Finally, as $\phi_0$ and $\phi_1$ are piecewise linear in $Z$, by Lemma~\ref{l:pw-linear-Z}, $\cT_0$ and $\cT_1$ can be taken to be simply connected in $Z$.

\begin{lemma}
	\label{Lem:LargeWinding}
	Let $i \in \set{0, 1}$ and suppose that $\psi\colon [0,1]\to B$ is an arc compatible with $\cR_i$. Then $(-1)^i w(\psi)\geq N/4-2$.
\end{lemma}

\begin{proof}
  We only consider the case $i = 0$. By Lemma~\ref{l:WindingBound} \ref{i:l:WB:2}, it suffices to show that $w(\phi_0|_{[1, N]})\geq N/4-1$. To show this, we first note that $w(\phi_0|_{[k, k+1]}) > 0$ for each $1\leq k \leq N-1$ and that $w(\phi_0|_{[k, k+2]}) > 1/4$ for each $1\leq k\leq N-2$. Second, we note that $w(\phi_0|_{[k, k+4q]})$ is within $1/4$ of an integer for any integer $q\geq 0$. From this, it follows that $w(\phi_0|_{[k, k+4q]})\geq q-1/4$, and the lemma follows.
\end{proof}

\begin{lemma}
	\label{LargeWindingPerturbation}
	Let $i \in \set{0, 1}$ and let $\psi\colon [0,1]\to B$ be an arc such that $\im{\psi}\subseteq \bigcup \mathcal{R}_i$ with $\psi(0)\in (R_i)_0$ and let $g\in G_\epsilon$. Then $(-1)^i w(g\cdot\psi) \geq -6$.
\end{lemma}
\begin{proof}
  We only consider the case $i = 0$. Suppose that $\psi(1)\in (R_0)_j$. We can find some $s\in [1, N]$ with $\phi_0(s)\in (R_0)_j$, and for some integer $k< N$, we have $s\in [k, k+1)$. Using the proof of Lemma~\ref{Lem:LargeWinding}, we have $w(\phi_0|_{[1, k]}) \geq k/4 - 2 \geq -2$. We also have $\phi_0([k, k+1])\subseteq \mathrm{Int}(B_{k-1})$, implying that $|w(\phi_0|_{[k, s]})| < 1/2$. Combining these and using Lemmas~\ref{l:WindingBound} and \ref{l:PerturbWinding} gives the result.
\end{proof}

We are finally in position to conclude the proof of Theorem~\ref{th:i:dimension3}. Towards a contradiction, suppose that for some $g\in G_\epsilon$, $O(\mathcal{T}_0)\cap g\cdot O(\mathcal{T}_1)\neq \emptyset$. By the density of rays in $M$, let $\rho \colon [0, \infty) \to X$ be a ray belonging to this intersection. Without loss of generality, we may assume that $\rho$ visits $(R_0)_{r_0-1}$ not later than $g \cdot (R_1)_{r_1-1}$ (otherwise, exchange $\cT_0$ and $\cT_1$ and consider $g^{-1}$ instead of $g$). This implies that there is an arc $\psi\colon [0,1]\to X$ (one can take a reparametrization of an initial segment of $\rho$) with the following properties:
\begin{itemize}
\item $\psi$ is compatible with $\cT_0$;
\item $\im \psi \sub \bigcup g \cT_1$.
\end{itemize}
We define $a\in [0,1]$ as follows: if $\im{\psi}\subseteq A$, set $a = 0$; otherwise, set $a = \sup(\psi^{-1}(X\setminus A))$. In any case, note that $a < 1$ and that $\psi(a)\in W_i$ for some $i< r$. Since $\psi([a, 1])\subseteq A$ and the links of $\cT_0$ have small diameter, we have that $\psi|_{[a,1]}$ is $(\cT_0, A')$-confined. Let $b = \sup(\psi^{-1}(X\setminus \bigcup \mathcal{R}_0))$ and note that $a < b < 1$. Also note that $\psi(b)\in W_{r-1}$.

Next we show that $|w(\psi|_{[a,b]})| \leq n+1$. First observe that there is a path $\theta\colon [0,1]\to W_i\cupdots W_{r-1}$ with $\theta(0) = \psi(a)$ and $\theta(1) = \psi(b)$. Then note that $W_i\cupdots W_{r-1}\subseteq V_j\cupdots V_{n-1}$ for some $j< n$ with $W_i\subseteq V_j$. Since each $V_q$ meets $A$ and has diameter at most $\epsilon$, we have $V_q\subseteq A'$, so that $\theta$ is $(\langle V_j,...,V_{n-1}\rangle, A')$-confined. Now we can apply Lemma~\ref{l:WindingBound} to conclude that $|w(\theta)|\leq n$. Next observe that both $\theta$ and $\psi|_{[a,b]}$ are $(\cT_0, A')$-confined and apply Lemma~\ref{l:WindingBound} again to conclude that $|w(\psi|_{[a,b]})| \leq n + 1$.

We now note that by Lemma~\ref{Lem:LargeWinding}, we have $w(\psi|_{[b,1]}) \geq N/4-2$. In particular, combining this with our previous observation, recalling that $N = 100 n$, we see that
\begin{equation}
  \label{eq:psi1}
  w(\psi|_{[a,1]}) \geq N/5.
\end{equation}

Next note that since $g^{-1} A' \sub Z$, we have that $g\cT_1$ is simply connected in $A'$. Let $c = \sup(\psi^{-1}(X\setminus \bigcup g \mathcal{R}_1))$. Note that as $\bigcup g \cR_1$ is at positive distance from $X \sminus A$, we must have $\psi(a)\not\in \bigcup g \mathcal{R}_1$ and $a < c\leq 1$. It follows that $\psi(a)\in gW_k$ for some $k < r$.

Next we show that $|w(\psi_{[a,c]})| \leq n + 1$ by an argument similar to the one above. There is a path $\eta\colon [0,1]\to gW_k\cupdots gW_{r-1}$ with $\eta(0) = \psi(a)$ and $\eta(1) = \psi(c)$. For some $\ell < n$ with $W_k\subseteq V_\ell$, we have $gW_k\cupdots gW_{r-1}\subseteq gV_\ell\cupdots gV_{n-1}$. Since each $gV_q$ touches $A$ and has diameter at most $3\epsilon$, we see that $\eta$ is $(\langle gV_\ell,...,gV_{n-1}\rangle, A')$-confined, so Lemma~\ref{l:WindingBound} implies that $|w(\eta)| \leq n$. As both $\eta$ and $\psi|_{[a,c]}$ are $(g\cT_1, A')$-confined, Lemma~\ref{l:WindingBound} yields that $|w(\psi|_{[a,c]})| \leq n+1$.

Finally, by Lemma~\ref{LargeWindingPerturbation}, we have $w(g^{-1} \psi|_{[c,1]}) \leq 6$. By Lemma~\ref{l:PerturbWinding}, $w(\psi|_{[c,1]}) \leq 9$. Hence
\begin{equation*}
  w(\psi|_{[a,1]}) = w(\psi|_{[a,c]}) + w(\psi|_{[c,1]}) < n + 1 + 9 < N/5,
\end{equation*}
contradicting \eqref{eq:psi1}. This concludes the proof of the first assertion of the theorem. The second follows from \cite{Angel2014}*{Proposition~14.1}.
\end{proof} 


\section{A comeager set of chains in the Hilbert cube}
\label{sec:comeager-set-chains}

In this section, we describe what the generic connected chain in the Hilbert cube looks like. We denote by $Q = [0, 1]^\N$ the Hilbert cube and by $\cC(Q)$ the space of maximal connected chains in $Q$. Denote by $E(Q)\subseteq V(Q)$ the compact space of closed, connected subsets of $Q$ (equipped with the Vietoris topology). Fix a compatible metric $d$ on $Q$.

Recall that a closed set $A \sub Q$ is called a \df{Z-set} if for every $\eps > 0$, there exists a continuous map $f \colon Q \to Q \sminus A$ such that for all $x \in Q$, $d(f(x), x) < \eps$. The following homogeneity property is a key fact about Z-sets in the Hilbert cube.

\begin{fact}[\cite{vanMill1989}*{Theorem 6.4.6}]
  \label{f:homeomorphism-extension}
  Let $A,B \sub Q$ be Z-sets. If $f \colon A \to B$ is a homeomorphism, then there exists $g \in \Homeo(Q)$ with $g|_A = f$.
\end{fact}

The following is also well-known.
\begin{prop}
  \label{p:Z-sets-Gdelta}
  The set $\set{A \in E(Q) : A \text{ is a Z-set}}$ is dense $G_\delta$ in $E(Q)$.
\end{prop}
\begin{proof}
  For any continuous map $f\colon Q\to Q$, the set $\{A\in V(Q): A\subseteq Q\setminus \im{f}\}$ is open in $V(Q)$. It follows that the collection of Z-sets is $G_\delta$ in $V(Q)$, so also in $E(Q)$. To show density, let $C\subset Q$ be a continuum and denote $\pi_{n} \colon Q\rightarrow[0,1]^{n}\times\{0\}\times\{0\}\times\cdots$ the projection map. It is clear that $\lim_n \pi_n(C) = C$ and by \cite{vanMill1989}*{Lemma 6.2.3(2)}, $\pi_{n}(C)$ is a Z-set.
\end{proof}

Next we recall several facts about the pseudo-arc. A \df{continuum} is a compact, connected space. An open cover $\set{U_{0},\ldots,U_{n-1}}$ of a continuum is called a \df{chain cover} if for all $i, j < n$, $U_{i}\cap U_{j}\neq\emptyset\iff|i-j|\leq1$. A continuum is \df{chainable} if every open cover admits a refinement that is a chain cover. A continuum is \df{indecomposable} if it is not the union of two proper subcontinua. It is \df{hereditarily indecomposable} if every subcontinuum is indecomposable. The \df{pseudo-arc} $\bP$ is the unique, up to homeomorphism, \df{non-degenerate} (having more than one point), chainable, hereditarily indecomposable continuum \cite{Bing1951}*{Theorem~1}. As hereditary indecomposability is obviously a hereditary property, it is clear that every non-degenerate subcontinuum of the pseudo-arc is again a pseudo-arc.

Another interesting fact about the pseudo-arc is that the map $\cC(\bP) \to \bP$ that associates to a chain its smallest element (which is a point) is a homeomorphism. Indeed, it is obviously continuous, and as by hereditary indecomposability, for any point $x \in \cP$, the set $\set{A \in E(\bP) : x \in A}$ is a chain, it is also injective.

The pseudo-arc has many other remarkable properties some of which we recall below.

\begin{fact}[Bing~\cite{Bing1948}]
  \label{f:pseudo-arc-homog}
  The pseudo-arc is \df{homogeneous}, i.e., its homeomorphism group acts on it transitively.
\end{fact}

The following is also well-known.
\begin{fact}[Bing~\cite{Bing1951}]
  \label{f:pseudo-arc-generic}
  The set $\set{A \in E(Q) : A \text{ is homeomorphic to the pseudo-arc}}$ is dense $G_\delta$ in $E(Q)$.
\end{fact}
Theorem~2 of \cite{Bing1951} is only stated for $\R^n$ ($n \geq 2$) and the Hilbert space but the proof is entirely local, so it also works for any manifold of dimension at least $2$ and the Hilbert cube.

\begin{lemma}
  \label{l:sub-pseudo-arcs}
  Let $\cU=\set{U_{i}:i<k}$ be a chain cover of the pseudo-arc $\bP$. Then for any $i\leq i'<k$, there exists a subcontinuum $B$ of $\bP$ such that $\set{j:U_{j}\cap B\neq\emptyset}=[i,i']$.
\end{lemma}
\begin{proof}
  It is well-known (see, for example, \cite{Lewis1999}*{1.7}) that $\bP$ can be embedded in the plane in such a way that there exists a sequence of tubes $\set{\cU^{n}}$ consisting of open balls such that $\set{U\cap \bP:U\in\cU^{0}}$ refines $\cU$ and $\cU^{n+1}$ \df{refines} $\cU^{n}$ (i.e., for every $V\in\cU^{n+1}$, there exists $U\in\cU^{n}$ such that $\cl{V}\sub U$) for every $n$, and we have that $\bP=\bigcap_{n}\bigcup\cU^{n}$.

Write $\cU^{n}=\set{U_{0}^{n},\ldots,U_{k_{n}}^{n}}$. Choose inductively a sequence of intervals $[i_{n},i_{n}']$ with $[i_{-1},i_{-1}']=[i,i']$ and $[i_{n+1},i_{n+1}']$ having the property that $\cl{U_{i_{n+1}}^{n+1}}\sub U_{i_{n}}^{n}$, $\cl{U_{i_{n+1}'}^{n+1}}\sub U_{i_{n}'}^{n}$, and $\bigcup_{i=i_{n+1}}^{i_{n+1}'}\cl{U_{i}^{n+1}}\sub\bigcup_{i=i_{n}}^{i_{n}'}U_{i}^{n}$. Finally set $B=\bigcap_{n}\bigcup_{i=i_{n}}^{i_{n}'}\cl{U_{i}^{n}}$.
\end{proof}

The next proposition describes a generic chain in $\cC(Q)$.
\begin{prop} \label{prop:generic-chain-z-set-pseudo-arcs} The set of chains $c \in \cC(Q)$ such that every $A \in c\sminus\set{Q}$ is a chainable, hereditarily indecomposable continuum (a point or a pseudo-arc) and a Z-set is dense $G_{\delta}$ in $\cC(Q)$.
\end{prop}
\begin{proof}
  Denote by $\cS$ the set of chains in the proposition. By Proposition~\ref{p:Z-sets-Gdelta} and Fact~\ref{f:pseudo-arc-generic}, we can write the set of chainable, hereditarily indecomposable Z-set continua (these include the points) in $E(Q)$ as $\bigcap_{n}O_{n}$ where each $O_{n}\sub E(Q)$ is open, dense and $O_{n}\supseteq O_{n+1}$ for all $n$.

Let $D_{n}\sub E(Q)$ be the complement of the closed set
\[
  \set{A \in E(Q) : \exists x\in Q \text{ with }d(x,A)\geq 1/n},
\]
so that $D_{n}$ is the open set of subcontinua of $Q$ that are $1/n$-dense. Let
\[
  Z_{n} = \set{c \in \cC(Q) : c\sub O_{n}\cup D_{n}}.
\]
We claim that each $Z_{n}$ is open and dense in $\cC(Q)$ and that $\cS=\bigcap_{n}Z_{n}$.

That $Z_{n}$ is open is clear, so we check density. As tubes form a $\pi$-base for $\cC(Q)$, it suffices to check that $Z_n$ intersects every tube. Let $\cU = \set{U_0, \ldots, U_{m-1}}$ be a tube. By prolonging and refining $\cU$ if necessary, we may assume that $O(U_0, \ldots, U_{m-1}) \sub D_n$. As the collection of Z-set pseudo-arcs is dense in $E(Q)$, there exists a pseudo-arc $P \in O(U_0, \ldots, U_{m-1})\cap \bigcap_n O_n$. Construct a sequence $P_{0}\sub P_{1}\sub\cdots\sub P_{m-1}$ of subcontinua of $P$ by inverse induction: take $P_{m-1}=P$ and if $P_{i}$ is defined, apply Lemma~\ref{l:sub-pseudo-arcs} to obtain $P_{i-1}\sub P_{i}$ such that $\set{j:P_{i-1}\cap U_{j}\neq\emptyset}=[0,i-1]$. Finally, we can take $P_{0}=\set{x_{0}}$ to be a point. By the remarks above, $c_0 = \set{K\in E(P) : x_{0}\in K}$ is a chain in $P$ compatible with the tube $\cU$. Now it suffices to extend $c_0$ arbitrarily to a maximal chain to obtain an element of $Z_{n} \cap O(\cU)$; it is in $Z_n$ since subsets of $P$ are $Z$-sets, so in $O_n$, and supersets of $P$ are $1/n$-dense, so in $D_n$.

Finally, we see that $\cS = \bigcap_{n}Z_{n}$. The $\sub$ inclusion is clear. For the other, if $c\in\bigcap_{n}Z_{n}$ and $A \in c \sminus \set{Q}$, then there exists $n_{0}$ such that for all $n\geq n_{0}$, $A\notin D_{n}$, so $A\in\bigcap_{n\geq n_{0}}O_{n}$ and is therefore a hereditarily indecomposable Z-set.
\end{proof}
\begin{remark}
  \label{rem:not-Hilbert-cube}
  Note that Proposition~\ref{prop:generic-chain-z-set-pseudo-arcs} remains true (with the same proof) if we replace the Hilbert cube with any manifold of dimension at least $2$ and we omit the requirement that the members of the chain be Z-sets.
\end{remark}

Next we consider the set of \df{partial chains} $\cP(Q)$. Recall from \cite{Gutman2008} that $\cC(Q)$ is the set of all maximal chains that are connected (as subsets of $V(Q)$). Thus we are led to define the set of partial chains as follows:
\begin{equation*}
  \cP(Q) = \set[\big]{c \in V^2(Q) : c \text{ is a chain}, c \text{ is connected, and } \exists x \in Q \text{ with } \set{x} \in c}.
\end{equation*}
Here, connected refers to $c$ as a subset of $V(Q)$. In other words, $\cP(Q)$ is the set of all initial segments of maximal chains. $\cP(Q)$ is a compact set (all of the three conditions in the definition are closed) and it is naturally a $\Homeo(Q)$-flow. It is not minimal as both $Q$ and $\cC(Q)$ are proper subflows. However, in contrast with $\cC(Q)$, it does have a comeager orbit.
\begin{prop}
  \label{p:com-orbit-partial-chains}
  The set of chains in $\cP(Q)$ whose largest element is a Z-set pseudo-arc is dense $G_\delta$ in $\cP(Q)$ and it constitutes a single $\Homeo(Q)$-orbit.
\end{prop}
\begin{proof}
  Denote by $\cS$ the set described in the proposition. Let the sets $O_n$ be defined as in the proof of Proposition~\ref{prop:generic-chain-z-set-pseudo-arcs}. Then it is clear that
  \begin{equation*}
    \cS = \bigcap_n \set{c \in \cP(Q) : c \sub O_n},
  \end{equation*}
  so $\cS$ is $G_\delta$ and it is also dense by the same argument as in  Proposition~\ref{prop:generic-chain-z-set-pseudo-arcs}.

  It remains to show that all elements of $\cS$ lie in a single orbit. Let $c_1, c_2 \in \cS$ and let $\set{x_1}, P_1, \set{x_2}, P_2$ be the smallest and the largest elements of $c_1, c_2$ respectively. By Fact~\ref{f:pseudo-arc-homog}, there exists a homeomorphism $f \colon P_1 \to P_2$ with $f(x_1) = x_2$. As $P_1$ and $P_2$ are both Z-sets, by Fact~\ref{f:homeomorphism-extension}, $f$ extends to a homeomorphism $g \in \Homeo(Q)$. Finally, as a maximal chain in a pseudo-arc is determined by its smallest element, we must have that $g \cdot c_1 = c_2$.
\end{proof}

\bibliographystyle{alpha}
\bibliography{umf-homeo}
	
\end{document}